\title{\vspace{-1cm}Musical chairs}
\author{Yehuda Afek \thanks{The Blavatnik School of Computer Science, Tel-Aviv University,
                Israel 69978. afek@tau.ac.il}
                 \and Yakov Babichenko \thanks{Department of Mathematics, Hebrew
                 University, Jerusalem 91904, Israel yak@math.huji.ac.il}
                 \and Uriel Feige \thanks{Department of Computer Science
and Applied Mathematics Weizmann Institute of Science Rehovot 76100, Israel.
uriel.feige@weizmann.ac.il. The author holds the
Lawrence G. Horowitz Professorial Chair at the Weizmann Institute. Work supported in part by The Israel
Science Foundation (grant No. 873/08).}
                 \and Eli Gafni \thanks{Computer Science Department, University of California, LA, CA 95024, eli@cs.ucla.edu }
                 \and Nati Linial \thanks{School of Computer Science and Engineering, Hebrew
                 University, Jerusalem 91904, Israel nati@cs.huji.ac.il}
                 \and Benny Sudakov \thanks{Department of Mathematics, UCLA, Los Angeles, CA, 90095. bsudakov@math.ucla.edu.
Research supported in part by NSF grant DMS-1101185 and by a USA-Israeli BSF grant.}
                 }

\date{}
\documentclass [11pt]{article}
\usepackage{graphicx}
\usepackage{amsfonts}
\usepackage{amsmath}
\newtheorem{thm}{Theorem}
\newtheorem{cor}[thm]{Corollary}

\newtheorem{lemma}[thm]{Lemma}
\newtheorem{definition}[thm]{Definition}
\newtheorem{proposition}[thm]{Proposition}

\newenvironment{proof}{\noindent\bf{Proof.}\rm}{\hfill$\Box$\bigskip}

\def\MM{m}

\begin{document}
\maketitle

\begin{abstract}
In the {\em Musical Chairs} game $MC(n,m)$ a team of $n$ players  plays
against an adversarial {\em scheduler}. The scheduler wins  if the game proceeds indefinitely, while termination after a finite number of rounds is declared a win of the team. At each round of the game each player
 {\em occupies} one of the $m$ available {\em chairs}. Termination (and a win of the team) is declared as soon as each player occupies a unique chair. Two players that simultaneously occupy the same chair are said to be {\em in conflict}. In other words, termination (and a win for the team) is reached as soon as there are no conflicts.  The only means of communication throughout the game is this: At every round of the game, the scheduler selects  an arbitrary nonempty set of players who are currently in conflict, and notifies each of them separately
that it must move. A player who is thus notified changes its chair according to its deterministic program.
As we show, for $m\ge 2n-1$ chairs the team has a winning strategy. Moreover, using topological arguments we show that this bound is tight. For
$m\leq 2n-2$ the scheduler has a strategy that is guaranteed to make the game continue indefinitely and thus win.
We also have some results on additional interesting questions. For example, if $m \ge 2n-1$ (so that the team can win), how quickly can they  achieve victory?
\end{abstract}

\newpage

\section{Introduction}
Communication is a crucial ingredient in every kind of
collaborative work. But what is the least possible amount of
communication required for a team of players to achieve certain goals?
Motivated by this question, we consider in this paper the following game.

The {\em Musical Chairs} game $MC(n,m)$  involves
$m$ {\em chairs} numbered
$1,\ldots,m$ and a team of $n$ {\em
players} $P_1,\ldots,P_n$, who are playing against
an adversarial {\em scheduler}. The scheduler's goal is to make the game run indefinitely in which case he wins. The termination condition is that each player settles in a different chair. Upon termination the team of players is declared the winner. We say that player $P$ is
{\em in conflict} if some other player $Q$
is presently occupying the same chair as $P$.
Namely, termination and a win of the team is reached if there are no
conflicts. The scheduler gets to decide, for each player $P_i$ the chair that $P_i$ occupies at the start of the game. As mentioned above we severely
restrict the communication between the players during the game. All such communication is mediated by the scheduler as follows: At every time step, and as long as there are conflicts, the scheduler selects an arbitrary nonempty set of players which are currently in conflict and notifies them
that they need to move. A player thus notified to be in conflict
changes its chair according to its deterministic program, which he chooses before the game.
During the game, each player has no information about the chairs of
other players beyond the occasional one bit that tells him that he must
move, and we insist that the choice of a player's next chair be deterministic. Consequently, a player's action depends only on its current chair, and the sequence of chairs that he had traversed so far.
Therefore the sequence of chairs that player $P_i$
traverses is simply an infinite word $\pi_i$ over the alphabet $1,\ldots,m$.
Recall that the adversary can start each player on an any of the chairs. Consequently we assume that each $\pi_i$ is {\em full}, i.e., it contains all the letters in $[m]$. So upon receiving a conflict notification from the
scheduler, player $P_i$ occupying chair $\pi_i[k]$ moves to chair
$\pi_i[k+1]$. The scheduler's freedom in choosing the players' initial chairs means that for every $i$ he selects an index $k_i$ and the game starts with each $P_i$ occupying chair $\pi_i[k_i]$. A winning strategy for the players is a choice of full words $\pi_i$ with the following property:
For every choice of initial positions $k_i$ and for every strategy of the scheduler the game terminates in a finite number of rounds, i.e., the players cannot be beaten by the scheduler.

In this paper we obtain several results about the musical chairs game. Our first theorem determines the  minimal $m$ for which the team of players wins the
$MC(n,m)$ game.

\begin{thm}
\label{thm:mainfull}
The team of players has a winning strategy in the $MC(n,m)$ game if and only if $m \ge 2n-1$.
\end{thm}
In the winning strategy that we produce, each word $\pi_i$ is {\em periodic}, or, what is the same, a finite word that $P_i$ traverses cyclically.
We also show that for every $N>n$ there exist $N$ full cyclic words on the alphabet $[m]=[2n-1]$ such that every set of
$n$ out of these $N$ words constitutes a winning strategy for the $MC(n,2n-1)$ game.

To prove the lower bound in Theorem \ref{thm:mainfull} we use Sperner's lemma (see, e.g., \cite{AZ}) a fundamental tool from combinatorial topology. The use of this tool in proving lower bounds for distributed algorithms was pioneered in
\cite{BG, HS99, SZ00}. If one is willing to assume a great deal of knowledge  in the field of distributed computing, it is possible to deduce our lower bound from known results in this area.
Instead, to make the paper self contained, we chose to include here a direct proof which we think is more illuminating and somewhat simpler than the one which would result from reductions to the existing literature.

Although the words in Theorem~\ref{thm:mainfull} use the least number of
chairs, namely $m=2n-1$, their {\em lengths} are
doubly exponential in $n$. This leads to several interesting questions.
Are there winning strategies for the MC game
with much shorter words, say of length $O(n)$? Perhaps even of length
$m$? Can we provide significantly better upper bounds on the number of rounds till termination? Even if the schedular is bound to lose the game, how long can he make the game last?
Our next two results give some answers to these questions.
Here we consider an $MC(n,m)$ {\em winning systems} with
$N$ words. This is a collection
of $N \geq n$ full words on $[m]$, every $n$ of which constitute a winning strategy for the players in the
$MC(n, m)$ game.

\begin{thm}
\label{thm:7n} For every $N \ge n$, almost every choice of $N$ words of length $cn\log N$ in an alphabet of $m=7n$ letters
is an $MC(n,m)$ winning system with $N$ full words. Moreover, every game on these words terminates in $O(n\log N)$ steps. Here $c$ is an
absolute constant.
\end{thm}

Since we are dealing with full words which we seek to make short, we are
ultimately led to consider the problem under the assumption that each (finite, cyclically traversed) word $\pi_i$ is a
{\em permutation} on $[m]$. We note that the context of distrubted computing offers no particular reason for this restriction and that we are motivated to study this
question due to its aesthetic appeal. We can design permutation-based
winning strategies for $MC(n,2n-1)$ game for very small $n$ (provably for
$n=3$, a computer assisted construction and proof for $n=4$). We suspect that no such
constructions are possible for large values of $n$, but we are
unable at present to show this.  We do know, though that
\begin{thm}
\label{thm:permutations} For every integer $d \ge 1$ there is an $MC(n,m)$ winning system with $N = n^d$ {\em permutations} on $m = cn$ symbols, where $c$ depends only on $d$. In fact, this
holds for almost every choice of $N$ permutations on $[m]$.
\end{thm}

We should stress that our proofs of Theorems~\ref{thm:7n}
and~\ref{thm:permutations} are purely existential. The explicit
construction of such systems of words remains largely open, though
we have the following result in this direction.

\begin{thm}
\label{thm:explicitperm} For every integer $d \ge 1$ there is an $MC(n,m)$ winning system with $N = n^d$ {\em permutations} on $m =  O(d^2 n^2)$ symbols.
\end{thm}

We conclude this introduction with a discussion of several additional aspects of the subject.

Our work was originally motivated by some questions in distributed computing.
In every distributed algorithm each processor must occasionally
observe the activities of other processors. This can be done
by reading the messages that other processors send, by inspecting some
publicly accessible memory cells into which they write, or by
sensing an effect on the environment due to the actions of other
processors. Hence it is very natural to ask: What is the least possible amount of
communication required to achieve certain goals? To answer it, we consider two severe
limitations on the processors' behavior and ask how this affects
the system's computational power. First, a processor can only post a
proposal for its own output, and second, each processor is
``blindfolded'' and is only occasionally provided with the least
possible amount of information, namely a single bit that indicates
whether its current state is ``good'' or ``bad''. Here
``bad/good'' indicates whether or not this state conflicts with
the global-state desired by all the processors. Moreover, we also
impose the requirement that algorithms are deterministic, i.e., use no
randomization. This new minimalist model, which we call the {\em oblivious model}, was introduced in the
conference version of this paper \cite{ABFGLS}. This model might appear to be
significantly weaker than other (deterministic) models studied in
distributed computing. Yet, our results show  that a very natural distributed problem
{\em musical chairs} \cite{GR05}, can be solved optimally within
the highly limited oblivious model. Further discussion of the oblivious model and
additional well-known problems like {\em renaming} \cite{AttiyaBDPR90,AAdaptiveRn} which we can also solve optimally in this model
can be found in \cite{ABFGLS}.

A winning strategy for the
$MC(n,m)$ game cannot include any two identical words. For that allows the scheduler to move the corresponding players together in
lock-step, keeping them constantly in a state of conflict. Also for every
winning strategy for $MC(n,m)$, with finite cyclic words, there is a finite
upper bound on the number of moves till
termination. To see this, let us associate with every state of the system a vector whose $i$-th coordinate is the current position of player $P_i$ on $\pi_i$. The set of such vectors $V$ is finite, $|V| = \prod |\pi_i|$, and in a terminating sequence of moves no vector can be
visited twice. In fact, we can associate with every collection of finite words a directed graph on vertex set $V$, where edges correspond to the possible transitions in response to scheduler's notifications. The
collection of words constitute a winning MC strategy iff this directed
graph is acyclic. We note that these observations depend on the assumption that players use no randomness.

Our strategies for the MC game have a number of additional
desirable properties. As mentioned, we construct $N$ full periodic words such that every subset of $n$ of the $N$ words constitutes an $MC(n,m)$ winning system. Hence our strategies are guaranteed to succeed (reach termination against every scheduler's strategy) in dynamic settings in which
the set of players in the system keeps changing. This statement holds provided there are
sufficiently long intervals throughout which the set of players remains unchanged.
To illustrate this idea, consider a
company that manufactures $N$ communication devices, each of which
can use any one of $m$ frequencies. If several such devices happen
to be at the same vicinity, and simultaneously transmit at the
same frequency, then interference occurs. Devices can move in
or out of the area, hop to a frequency of choice and transmit
at this frequency, sense whether there are other
transmissions in this frequency. The company wants to provide the
following guarantee: If no more than $n$ devices reside in the
same geographical area, then no device will suffer more than a total of $T$
interference events for some guaranteed bound $T$. Our strategy for the
MC game would yield this by pre-installing in each
device a list of frequencies (a {\em word} in our terminology),
and having the device hop to the next frequency on its list (in a
cyclic fashion) in response to any interference it encounters. No
communication beyond the ability to sense interference is needed.

In proving the lower bound $m \ge 2n-1$ we have to make several assumptions about the setup. The first
is the freedom of choice for the scheduler. From the perspective of distributed computing this means that we are dealing with
an asynchronous system. In a synchronous setting, in every time step, every player involved in a conflict
moves to its next state. One can show that in such a synchronous setup the players have a winning strategy even with $m = n$ chairs.
It is also important that the scheduler can dictate each player's starting position. If each $P_i$ starts at the first letter of $\pi_i$,
a trivial winning strategy with $m=n$ simply sets $i$ as the first letter of $\pi_i$ for each $i$. It is also crucial that
our players are deterministic (no randomization). If players are allowed to pick their next state randomly, then again
$m = n$ suffices, since in this case with probability~1 eventually a conflict-free configuration
will be reached. Hence, this paper is also related to the one of the fascinating questions in computer science, whether and to what extent
randomization increases the power of algorithmic procedures. Our results show that
without using randomness one can still win an MC game by increasing only slightly the number of chairs (from $n$ to $2n-1$).

\section{Simplified Oblivious Model for Musical Chairs}

Our general model for oblivious algorithms is specified by providing the rules of possible behavior of the scheduler. Here we consider an
{\em immediate scheduler}, who enjoys a high degree of freedom in choosing which
processor to move. To simplify the design and analysis of oblivious
algorithms, it is convenient to consider a more restricted scheduler that
has fewer degrees of freedom, but is nevertheless equivalent to
the immediate scheduler in their power to win the MC game. In each round an immediate scheduler can select an
arbitrary nonempty set of players that are currently in conflict and move them. Below we often refer to a team strategy as an oblivious
$MC(n,m)$ algorithm. It is a winning strategy if the immediate scheduler is forced to reach a conflict-free configuration in finite time. Conversely, an immediate scheduler wins against an oblivious $MC(n,m)$ algorithm if it can generate an infinite execution without ever reaching a conflict-free configuration.

{\bf Terminology.} Two schedulers $\sigma_1, \sigma_2$ are considered {\em
equivalent} if for every team strategy scheduler $\sigma_1$ has a winning strategy  iff so does $\sigma_2$.

First we want to limit the number of processors that can be moved in a round.

A {\em pairwise immediate scheduler}  is similar to the immediate scheduler, except for the following restriction. In every round, the pairwise immediate scheduler can select any two processors $P \neq Q$ that are currently in conflict with each other, and move either $P$, or $Q$, or both. Equivalently, in every round
either exactly one processor (that is involved in a conflict) moves, or two processors that share the same chair move.

\begin{proposition}
The immediate scheduler and the pairwise immediate scheduler are equivalent.
\end{proposition}

\begin{proof}
The pairwise immediate scheduler is more restricted than the immediate scheduler. Hence it remains to show that if the immediate scheduler can win against some team strategy, then the pairwise immediate scheduler can also force an infinite run against it.
We prove this last statement by a double induction on the round $t$ and the number $k$ of processors that move in round $t$.

Let us fix an oblivious $MC(n,m)$ algorithm and an infinite run that is forced by the immediate scheduler. Let $t$ be the first round in which the moves are not
consistent with any pairwise immediate scheduler and let $k$ be the number of processors that move in round $t$. There are two cases to consider. In one case $k \ge 3$ and all the moving processors occupy the same chair in round $t$. Break round $t$ into two rounds, pushing future rounds by one. In the first of
them (round $t$) move only one of the processors from $S$ and in the second round (round $t+1$) move the rest (they can still move because there are at least two
of them). This completes the inductive step with respect to $t$. The other case is that the set $S$ of processors that moved in round $t$ collided on at least two
different chairs. Pick one of these chairs, say chair $c$, and let $S(c)$ be the set of those processors in $S$ that in round $t$ occupy chair $c$. Break
round $t$ into two rounds, pushing future rounds by one. In the first of them (round $t$) move only those processors in $S(c)$, and in the second round (round
$t+1$) move those processors in $S - S(c)$. This completes the inductive step (as either $k$ decreased or $t$ increased).
\end{proof}

The use of the pairwise immediate scheduler (which as we showed is equivalent to our original scheduler) helps simplify the proofs of
theorems~\ref{thm:mainfull} and~\ref{thm:permutations}. However, for the proof of Theorem~\ref{thm:7n} even the pairwise immediate
scheduler should be further restricted. It is true that it has to pick only one pair of players to move (and then either move only one or both of them), but
it is still free to pick a pair of its choice (among those pairs that are in conflict). We would like to eliminate this degree of freedom.

{\em Canonical Scheduler}. The canonical scheduler is similar to the pairwise immediate scheduler but with the following difference. In every round in which there
is a conflict, one designates a {\em canonical pair}. This is a pair of players currently in conflict with each other, but they are not chosen by the
scheduler, but rather dictated to the scheduler. Given the canonical pair $P, Q$, the only choice the scheduler has is
whether to move $P$, or $Q$, or both. But
how is the canonical pair chosen? In the current paper this does not really matter to us, as long as the choice is deterministic. For concreteness, we shall
assume the following procedure. Fix an arbitrary order on the collection of all pairs of players. In a nonterminal configuration, the {\em canonical pair} is the first pair of players in the order that share a chair.

\begin{proposition}
The canonical scheduler and the pairwise immediate scheduler are equivalent.
\end{proposition}

\begin{proof}

Consider an oblivious $MC(n,m)$ algorithm and an infinite run against the immediate scheduler. Let $t$ be the first round in which the run is
inconsistent with a canonical scheduler. That is, the canonical pair at round $t$ consists of two players, say $P_1$ and $P_2$, that
occupy the same chair, say chair $c_1$, whereas the immediate scheduler moved at least one player not from the canonical pair. We consider several cases.

{\em Case 1.} The immediate scheduler never moves $P_1$ in any round from $t$ onwards. In this case move $P_2$ in round $t$.  Note that all moves (except for the
move just performed, moving $P_2$ away from $c_1$) performed by the immediate scheduler from round $t$ onwards are still available to this scheduler (because
chair $c_1$ remains occupied). Hence the total number of moves in the schedule did not change, whereas $t$ increases by one, completing the inductive step. The
same argument can be applied with $P_1$ and $P_2$ exchanged.

{\em Case 2.} The immediate scheduler moves $P_2$ away from $c_1$ in a later round than it moves $P_1$. In this case move $P_1$ in round $t$. Again, all moves
(except for the move just performed, moving $P_1$ away from $c_1$) performed by the immediate scheduler from round $t$ onwards are still available to this
scheduler. The same argument can be applied with $P_1$ and $P_2$ exchanged.

{\em Case 3.} The immediate scheduler moves both $P_1$ and $P_2$ out of $c_1$ in the same round $t' \ge t$. There are two subcases to consider. In one, there is
no player other than $P_1$ and $P_2$ on chair $c_1$ in any of the rounds $t, \ldots, t'$. In this subcase, move $P_1$ and $P_2$ in round $t$ (pushing future
rounds by one). All moves performed by the immediate scheduler from round $t$ to $t'$ are still available to this scheduler. The other subcase is that there is
some round $t \le t" \le t'$ in which some other player say $P_3$ is on chair $c_1$. Consider the largest such $t"$. Move $P_1$ in round $t$ (pushing future
rounds by one) and $P_2$ in round $t"+1$ (the round that previous to the pushing of rounds was round $t"$), together with whoever else is moved at that round. \end{proof}

\noindent
{\bf Remark.} There are several interesting schedulers which are even more flexible than the immediate scheduler. As we showed in the conference version of this paper \cite{ABFGLS} all of them are, in fact, equivalent to the canonical scheduler.

\section{An oblivious MC algorithm with $2n-1$ chairs}

\label{sec:OMC}

\subsection{Preliminaries}
\label{prelim}

In this section we prove the upper bound that is stated in
Theorem~\ref{thm:mainfull}. We start with some preliminaries. The
length of a word $w$ is denoted by $|w|$. The concatenation of
words is denoted by $\circ$. The $r$-th power of $w$ is denoted by
$w^r = w \circ w \ldots \circ w$ ($r$ times). Given a word $\pi$
and a letter $c$, we denote by $c\otimes \pi $ the word in which
the letters are alternately $c$ and a letter from $\pi$ in
consecutive order. For example if $\pi =2343$ and $c=1$ then
$c\otimes \pi =12131413$. A collection of words $\pi _{1},\pi
_{2},...,\pi _{n}$ is called \textit{terminal} if no schedule can
fully traverse even one of the $\pi _{i}$. Note that we can
construct a terminal collection from any $MC$ algorithm just by
raising each word to a high enough power.

We now introduce some of our basic machinery in this area. We
first show how to extend terminal sets of words.

\begin{proposition}
\label{pro:full-equivalence} Let $n, m, N$ be integers with $1 < n
< m$. Let $\Pi = \{\pi_1, \ldots \pi_N\}$ be a collection of
$m$-full words such that
\begin{equation}
\label{every_n} \mbox{every~} n \mbox{~of these words form an
oblivious~}MC(n,m) \mbox{~algorithm}.
\end{equation}
Then $\Pi$ can be extended to a set of $N+1$ $m$-full words that
satisfy condition~(\ref{every_n}).
\end{proposition}

\begin{proof}
Suppose that for every choice of $n$ words from $\Pi$ and for
every initial configuration no schedule lasts more than $t$ steps.
(By the pigeonhole principle $t \le L^n$, where $L$ is the length
of the longest word in $\Pi$). For a word $\pi$, let $\pi'$ be
defined as follows: If $|\pi| \ge t$, then $\pi'=\pi$. Otherwise
it consists of the first $t$ letters in $\pi^r$ where $r >
|\pi|/t$. The new word that we introduce is $\pi_{N+1} = \pi_1'
\circ \pi_2' \circ \ldots \circ \pi_n'$. It is a full word, since
it contains the full word $\pi_1$ as a sub-word.

We need to show that every set $\Pi'$ of $n-1$ words from $\Pi$
together with $\pi_{N+1}$ constitute an oblivious $MC(n,m)$
algorithm. Observe that in any infinite schedule involving these
words, the word $\pi_{N+1}$ must move infinitely often. Otherwise,
if it remains on a letter $c$ from some point on, replace the word
$\pi_{N+1}$ by an arbitrary word from $\Pi - \Pi'$ and stay put on
the letter $c$ in this word. This contradicts our assumption
concerning $\Pi$. (Note that this word contains the letter $c$ by
our fullness assumption.) But $\pi_{N+1}$ moves infinitely often,
and it is a concatenation of $n$ words whereas $\Pi'$ contains
only $n-1$ words. Therefore eventually $\pi_{N+1}$ must reach the
beginning of a word $\pi_{\alpha}$ for some $\pi_{\alpha} \not\in
\Pi'$. From this point onward, $\pi_{N+1}$ cannot proceed for $t$
additional steps, contrary to our assumption.
\end{proof}

Note that by repeated application of
Proposition~\ref{pro:full-equivalence}, we can construct an
arbitrarily large collection of $m$-full words that satisfy
condition~(\ref{every_n}).

We next deal with the following situation: Suppose that $\pi
_{1},\pi _{2},...,\pi _{m}$ is a terminal collection, and we
concatenate an arbitrary word $\sigma $ to one of the words
$\pi_{i}$. We show that by raising all words to a high enough
power we again have a terminal collection in our hands.

\begin{lemma}
\label{basic_lemma} Let $\pi _{1},\pi _{2},...,\pi _{p}$ be a
terminal collection of full words over some alphabet. Let $\sigma$
be an arbitrary full word over the same alphabet. Then the
collection
\begin{eqnarray*}
(\pi _{1})^{k},(\pi _{2})^{k},...,(\pi _{i-1})^{k},(\pi _{i}\circ
\sigma )^{2},(\pi _{i+1})^{k},...,(\pi _{p})^{k}
\end{eqnarray*}
is terminal as well, for every $1\leq i\leq p$, and every $k\geq
|\pi _{i}|+|\sigma |$.
\end{lemma}

\begin{proof}
We split the run of any schedule on these words into {\em periods}
through which we do not move along the word $(\pi _{i}\circ
\sigma)^{2}$. We claim that throughout a single period we do not
traverse a full copy of $\pi _{j}$ in our progress along the word
$(\pi _{j})^{k}$. The argument is the same as in the proof of
Proposition~\ref{pro:full-equivalence}. By pasting all these
periods together, we conclude that during a time interval in which
we advance $\leq |\pi _{i}|+|\sigma |-1$ positions along the word
$(\pi _{i}\circ \sigma )^{2}$ every other word $(\pi _{j})^{k}$
traverses at most $|\pi _{i}|+|\sigma |-1$ copies of $\pi _{j}$.
In particular, there is a whole $\pi _{j}$ in the $j$-th word in
the collection that is never visited. If the schedule ends in this
way, no word is fully traversed, and our claim holds.

So let us consider what happens when a schedule makes $\geq |\pi
_{i}|+|\sigma |$ steps along the word $(\pi _{i}\circ \sigma
)^{2}$. We must reach at some moment the start of $\pi _{i}$ in
our traversal of the word $(\pi _{i}\circ \sigma )^{2}$. But our
underlying assumption implies that from here on, no word can fully
traverse the corresponding $\pi_{k}$ (including $\pi_{i}$). Again,
no word is fully traversed, as claimed.
\end{proof}

Lemma \ref{basic_lemma} yields immediately:

\begin{cor}
\label{cor_2} Let $\pi _{1},\pi _{2},...,\pi _{p}$ be a terminal
collection of full word over some alphabet, and let
$\pi_{p+1},$ $\pi_{p+2},...,\pi _{n}$ be arbitrary full words over
the same alphabet. Then the collection
\begin{eqnarray*}
(\pi _{1}\circ \pi _{2}\circ ...\circ \pi _{n})^{2},(\pi
_{1})^{k},(\pi _{2})^{k},...,(\pi _{i-1})^{k},(\pi
_{i+1})^{k},...,(\pi _{p})^{k}
\end{eqnarray*}%
is terminal as well. This holds for every $1\leq i\leq p$ and
$k\geq \sum^{n}_{i=1}|\pi _{i}|$.
\end{cor}

This is a special case of Lemma~\ref{basic_lemma} where
$\sigma=\pi _{i+1}\circ \ldots \pi _{n}\circ \pi _{1}\ldots \circ
\pi _{i-1}$.

\subsection{The MC($n,2n-1$) upper bound}
\label{app:upperbound}

The proof we present shows somewhat more than
Theorem~\ref{thm:mainfull} says. A useful observation is that the scheduler
can ``trade'' a player $P$ for a chair $c$. Namely, he can keep
$P$ constantly on chair $c$ and be able, in return to move any
other player past $c$-chairs. In other words, this effectively
means the elimination of chair $c$ from all other words. This
suggests the following definition: If $\pi$ is a word over
alphabet $C$ and $B\subseteq C$, we denote by $\pi (B)$ the word
obtained from $\pi$ by deleting from it the letters from
$C\setminus B$.

Our construction is recursive. An inductive step should add one
player (i.e., a word) and two chairs. We carry out this step in
two installments: In the first we add a single chair and in the
second one we add a chair and a player. Both steps are accompanied
by conditions that counter the above-mentioned trading option.

\begin{proposition}
\label{main_prep} For every integer $n\geq 1$
\begin{itemize}
\item
There exist full words $s_{1},s_{2},...,s_{n}$ over the alphabet
$\{1,2,...,2n-1\}$ such that\\
$s_{1}(A),s_{2}(A),...,s_{p}(A)$ is a terminal collection for
every $p\leq n$, and every subset\\
$A \subseteq \{1,2,...,2n-1\}$ of cardinality $|A|=2p-1$.
\item
There exist full words $w_{1},w_{2},...,w_{n}$ over alphabet
$\{1,2...,2n\}$, such that\\
$w_{1}(B),w_{2}(B),...,w_{p}(B)$ is a terminal collection for
every $p\leq n$, and every subset\\
$B\subseteq \{1,2,...,2n\}$ of cardinality $|B|=2p-1$.
\end{itemize}
\end{proposition}

The words $s_{1},s_{2},...,s_{n}$ in Proposition~\ref{main_prep}
constitute a terminal collection and are hence an oblivious
$MC(n,2n-1)$ algorithm that proves the upper bound part of Theorem
\ref{thm:mainfull}. In the rest of this section we prove
Proposition~\ref{main_prep}.


\begin{proof}

As mentioned, the proof is by induction on $n$. For $n=1$ clearly
$s_{1}=11$ and $w_{1}=1122$ satisfy the conditions.

In the induction step we use the existence of
$s_{1},s_{2},...,s_{n}$ to construct $w_{1},w_{2},...,w_{n}$.
Likewise the construction of $s_{1},s_{2},...,s_{n+1}$ builds on
the existence of $w_{1},w_{2},...,w_{n}$.

\textbf{The transition from }$w_{1},w_{2},...,w_{n}$ \textbf{to }
$s_{1},s_{2},...,s_{n+1}$\textbf{:}

To simplify notations we assume that the words
$w_{1},w_{2},...,w_{n}$ in the alphabet $\{2,3,...,2n+1\}$ (rather
than $\{1,2,...,2n\}$) satisfy the proposition. Let $k:=\sum
|w_{i}|$ and define:

\begin{eqnarray*}
s_{1} &:&=1\otimes ((w_{1}\circ w_{2}\circ ...\circ w_{n})^{2(2n+1)}) \\
\forall i=2,\ldots n+1~~~ s_{i} &:&=(w_{i-1})^{k(2n+1)}\circ 1
\end{eqnarray*}

Fix a subset $A\subseteq \{1,2,...,2n+1\}$ of cardinality
$|A|=2p-1$ with $p \le n+1$, and let us show that
$s_{1}(A),s_{2}(A),...,s_{p}(A)$ is a terminal collection. There
are two cases to consider:

We first assume ${1\notin A}$. This clearly implies that $p\leq n$
(or else $A=\{1,2,...,2n+1\}$ and in particular $1\in A$). In this
case the collection is:

\begin{eqnarray*}
s_{1}(A) &:&= ((w_{1}(A)\circ w_{2}(A)\circ ...\circ w_{n}(A))^{2(2n+1)}) \\
\forall i=2,\ldots p~~~ s_{i}(A) &:&=(w_{i-1}(A))^{k(2n+1)}
\end{eqnarray*}

By the induction hypothesis, the collection
$w_{1}(A),w_{2}(A),...,w_{p-1}(A),w_{p}(A)$ is terminal. We apply
Corollary \ref{cor_2} and conclude that
\begin{eqnarray*}
(w_{1}(A)\circ w_{2}(A)\circ ...\circ
w_{n}(A))^{2},(w_{1}(A))^{k},(w_{2}(A))^{k},...,(w_{p-1}(A))^{k}
\end{eqnarray*}
is terminal as well. But the $s_{i}$ are obtained by taking
$(2n+1)$-th powers of these words, so that
$s_{1}(A),s_{2}(A),...,s_{p}(A)$ is terminal as needed.

We now consider what happens when ${1\in A}$.

We define $F_{1}:=(w_{1}(A)\circ w_{2}(A)\circ ...\circ
w_{n}(A))^{2}$ and for for $j>1$, let $F_{j}:=(w_{j-1}(A))^{k}$. We
refer to $F_{i}$ as the $i$-th block. In our construction each word
has $2n+1$ blocks, ignoring chair $1$.

At any moment throughout a schedule we denote by ${\cal O}_{1}$
the set of players in $\{P_{2},P_{3},...,P_{p}\}$ that currently
occupy chair $1$. We show that during a period in which the set
${\cal O}_1$ remains unchanged, no player can traverse a whole
block. The proof splits according to whether ${\cal O}_1$ is empty
or not.

Assume first that ${\cal O}_1\neq \emptyset$, and pick some $i>1$
for which $P_{i}$ occupies chair $1$ during the current period. As
long as ${\cal O}_1$ remains unchanged, $P_{i}$ stays on chair
$1$, so the words that the other players repeatedly traverse are
as follows: For $P_1$ it is

\[
w_{1}(A\backslash \{1\})\circ w_{2}(A\backslash
\{1\})\circ...\circ w_{n}(A\backslash \{1\})
\]

and for $P_j$ with $p \ge j \neq i \ge 2$ it is

\[
w_{j-1}(A\backslash \{1\})
\]

We now show that no player can traverse a whole block (as defined
above). Observe that the collection $\{w_{\nu}(A\backslash
\{1\})|\nu=1,\ldots,p-1\}$ (including, in particular the word
$w_{i-1}(A\backslash \{1\})$) is terminal. This follows from the
induction hypothesis, because $|A\backslash \{1\}|=2p-2$, and
because the property of being terminal is maintained under the
insertion of new chairs into words. Applying Corollary~\ref{cor_2}
to this terminal collection implies that this collection of blocks
is terminal as well.

We turn to consider the case ${\cal O}_1=\emptyset$. In this case
player $1$ cannot advance from a none-$1$ chair to the next
none-$1$ chair, since the two are separated by the presently
unoccupied chair $1$. We henceforth assume that player $P_{1}$
stays put on chair $c\neq 1$, but our considerations remain valid
even if at some moment player $P_{1}$ moves to chair $1$. (If this
happens, he will necessarily stay there, since ${\cal
O}_1=\emptyset$). We are in a situation where players
$P_{2},P_{3},...,P_{p}$ traverse the words $w_{1}(A\backslash
\{1,c\}),w_{2}(A\backslash \{1,c\}),...,w_{p-1}(A\backslash
\{1,c\})$ (chair $c$ which is occupied by player $P_{1}$ can be
safely eliminated from these words). But $|A\backslash
\{1,c\}|=2p-3$, so by the induction hypothesis no player can
traverse a whole $w_{i}(A\backslash \{1,c\})$, so no player can
traverse a whole block.

We just saw that during a period in which the set ${\cal O}_1$
remains unchanged, no player can traverse a whole block.

Finally, assume towards contradiction that $P_{j}$ fully traverses
$s_{j}$ for some index $j$, and consider the first occurrence of
such an event. It follows that $P_{j}$ has traversed $2n+1$
blocks, so that the set ${\cal O}_1$ must have changed at least
$2n+1$ times during the process. However, for ${\cal O}_1$ to
change, some $P_{i}$ must either move to, or away from a 1-chair
in $s_{i}$. But $1$ occurs exactly once in $s_{i}$, so every
$P_{i}$ can account for at most two changes in ${\cal O}_1$, a
contradiction.

\textbf{The transition from }$s_{1},s_{2},...,s_{n}$\textbf{\ to}
$w_{1},w_{2},...,w_{n}$\textbf{:}

We assume that the words $s_{1},s_{2},...,s_{n}$ in the alphabet
$\{2,3,...,2n\}$ satisfy the proposition. Let $k:=\sum |s_{i}|$
and define:

\begin{eqnarray*}
w_{1} &:&=1\otimes ((s_{1}\circ s_{2}\circ ...\circ s_{n})^{2(2n+1)}) \\
\forall i=2,\ldots,n~~~ w_{i} &:&=(s_{i-1})^{k(2n+1)}\circ 1
\end{eqnarray*}

Fix a subset $B\subseteq \{1,2,...,2n\}$ with $|B|=2p-1$. Then

\begin{eqnarray*}
w_{1}(B) &=&1\otimes ((s_{1}(B)\circ s_{2}(B)\circ ...\circ s_{n}(B))^{2(2n+1)}) \\
\forall i=2,\ldots,p~~~ w_{i}(B) &=&(s_{i-1}(B))^{k(2n+1)}\circ 1
\end{eqnarray*}

\noindent
are exactly the same as in the previous transition just by
replacing $s$ with $w$ and $A$ with $B$ (in this case the
induction hypothesis is on $s_{i} $ and we prove for $w_{i}$). So
exactly the same considerations prove that
$w_{1}(B),w_{2}(B),...,w_{m}(B)$ is a terminal collection.
\end{proof}

\section{Impossibility results}

In this section we prove the lower bound of Theorem~\ref{thm:mainfull}. As it turns out, the situation for $2n-2\ge m$ and for $m \ge 2n-1$  are dramatically different. As we saw, for $m \ge 2n-1$ the team has a winning strategy, but when $2n-2 \ge m$ not only is it true that the scheduler can win the game. He is guaranteed to have a winning startegy even if we (i) substantially relax the requirement that each word $\pi_i$ over $[m]$ be full, or (ii) restrict his power to select the players' starting position on their words. In the next proposition case (i) occurs:

\begin{proposition}\label{prop:clearly_implies}
Every team strategy $\tau_1,\ldots,\tau_n$ over $[m] = [2n-2]$ for which:
\begin{itemize}
\item
The chair $1$ appears in both $\tau_1, \tau_2$, and
\item
For every $3 \le i \le n$, the word $\tau_i$ contains both chair $2i-4$ and $2i-3$
\end{itemize}
is a losing strategy.
\end{proposition}

Needless to say, this statement is invariant under permuting the player's names and the indices of the chairs. There are several such arbitrary choices of indices below and we hope that this creates no confusion.
In the impossibility results that we prove in this section, the number of chairs $m$ is always $2n-2$. We also go beyond
the lower bound of Theorem~\ref{thm:mainfull} by considering scenarios with a total of $N \ge n$
players and statements showing that there
is a choice of $n$ out of the $N$ words that constitute a losing strategy. (Clearly, new words that get added to a losing team strategy make it only easier for the scheduler to win). These deviations from the basic setup
($N \ge n$ words, weakened fullness conditions, starting points not controlled by the scheduler) give us more flexibility in our arguments and complement each other nicely.
Here is one of the main theorems that we prove in this section. It yields exponentially many subsets of $n$ words that constitute a losing team strategy.

\begin{thm}
\label{thm:mainLB} Let $N = 2n-2$ and let $\pi_1,\ldots,\pi_N$ be words over $[m]=[2n-2]$ such that the only equality among the symbols $\pi_1[1], \pi_2[1], \pi_3[1],\ldots,\pi_N[1]$ is $\pi_1[1]=\pi_2[1]$.
Then, for every partition of the words $\pi_3,\ldots,\pi_N$
into $n-2$ pairs, there is a choice of one word from each pair, such
that the chosen words together with $\pi_1$ and $\pi_2$ constitute a losing team strategy even when the
game starts on each word's first letter.\end{thm}

While it is obvious that Proposition \ref{prop:clearly_implies} yields the lower bound of Theorem~\ref{thm:mainfull}, it is not entirely clear how Theorem \ref{thm:mainLB} fits into the picture. We show next how to derive Proposition \ref{prop:clearly_implies} from Theorem \ref{thm:mainLB}.

\begin{proof}[Theorem \ref{thm:mainLB} implies Proposition \ref{prop:clearly_implies}]
Let $\pi_1$ (resp. $\pi_2$) be the suffix of $\tau_1$ (resp. $\tau_2$) starting with the first appearance of the symbol $1$.
The other words come in pairs.
 For $3 \le i \le n$, we define $\pi_{2i-4}$ to be the suffix of $\tau_i$ starting at chair $2i-4$, and
$\pi_{2i-3}$ is its suffix starting at
chair $2i-3$. Theorem~\ref{thm:mainLB} implies that there is a
choice of one word from each pair that together with $\pi_1$ and
$\pi_{2}$ is losing when started from the initial chairs. The same scheduler strategy
clearly wins the game on $\tau_1,\ldots,\tau_n$ when started from
the respective chairs.
\end{proof}

The proof of Theorem \ref{thm:mainLB}, which uses some simple topological methods, is presented
in Section~\ref{sec:appLBproof}. We provide all the necessary background material
for this proof in Section~\ref{sec:appsperner}.

What happens if the fullness condition is eliminated altogether
but the scheduler maintains his right to select the starting positions? The scheduler clearly loses against the words
$\pi_i = (i)$ for $i=1,\ldots,m$. However, as the following
theorem shows once $N > m = 2n-2$, the scheduler has a winning strategy.

\begin{thm}
\label{thm:LBN} For every collection of $N=2n-1$ words over $[m] = [2n-2]$, there is a choice of $n$ words and starting
locations for which the scheduler wins.
\end{thm}

\begin{proof}
By pigeonhole, the scheduler wins agianst every set of words $S$ that together contain fewer than $|S|$ different letters. If such a subcollection $S$ exists with $|S| \le n$ we are clearly done. So consider such an $S$ of smallest cardinality. By assumption $|S| >n$.
 By minimality, the total number of letters that appear in the words of $S$ is exactly
$|S|-1$. By the Marriage Theorem, for every word $\pi \in S$ it is possible to mark one letter in every word in $S \setminus \{\pi\}$ so that every letter gets marked at most once. Let $S'$ consist of $\pi$ and the suffix of every other words in $S \setminus \{\pi\}$  starting from the marked letter. If $|S|=|S'|$ is even, then Theorem \ref{thm:mainLB} applies since $S'$ has more words than letters and there is exactly one coincidence among these words' initial letters. Consequently, $S$ has a subcollection of $\frac{|S|}{2}+1\le n$ that is a losing team stratgey, as claimed. If $|S|$ is odd we first delete a word from $S$ whose marked letter differs from $\pi[1]$ and argue as above.
\end{proof}

\subsection{A few words on Sperner's lemma}
\label{sec:appsperner}

In this section we discuss our main topological tool,
the Sperner Lemma (see, e.g., \cite{AZ}). We include all the required background and try to keep our presentation
to the minimum that is necessary for a proof of Theorem~\ref{thm:mainLB}.

\begin{definition}
A {\em simplicial complex} is a collection $X$ of subsets of a finite set of {\em vertices} $V$ such that
\[
\mbox{If}~~A \in X~~\mbox{and}~~B \subseteq A~~\mbox{then}~~B \in X.
\]
A member $A \in X$ is called a {\em face}, and its {\em dimension} is defined as $\dim A :=|A|-1$. We refer to $d$-dimensional faces as $d$-faces, and define $\dim X$ as the largest dimension of a face in $X$. We note that a vertex is a $0$-face, and call a $1$-face an {\em edge}. The $1$-{\em skeleton} of $X$ is the graph with vertex set $V$, where $xy$ is an edge iff $\{x, y\}$ is an edge ($1$-face) of $X$. A face of dimension $\dim X$ is called a {\em facet}. We say that $X$ is {\em pure} if every face of $X$ is contained in {\em some} facet. Finally, a $d$-{\em pseudomanifold} is a pure $n$-dimensional complex $X$ such that
\begin{equation}\label{pseudom}
\mbox{Every face of dimension~} d-1 \mbox{~is contained in exactly two facets.}
\end{equation}
\end{definition}

A good simple example of a 2-dimensional pseudomanifold is provided by a planar graph in which all
faces including the outer face are triangles. The vertices and the edges of the complex
are just the vertices and the edges
of the graph. The facets (2-simplices) of the complex are the faces of the planar graph,
including the outer face. This is clearly a pure complex and every edge is contained in exactly two facets.
Note that such a graph drawn on a torus or on another 2-manifold works just as well. The {\em pseudo} part of the definition comes since we are allowing to carry out identifications such as the following: Take a set of vertices that forms an anticlique in the graph and identify all of them to a single vertex. The result is still a 2-dimensional psudomanifold. At any event, the uninitiated reader is encoutaged to use planar triangulations as a good mental model for a pseudomanifold. Henceforth we shorten pseudomanifold to {\em psm}.

Let $X$ be a psm on vertex set $V$.
A $k$-{\em coloring} of $X$ is a mapping $\varphi: V \to \{1, \ldots, k\}$. A face of $X$ on which $\varphi$ is $1:1$ is said to be $\varphi$-{\em rainbow} (we only say rainbow, when it is clear what coloring is involved). We are now ready to state and prove a special case of
Sperner's lemma that suffices for our purposes.

\begin{lemma}
\label{lem:sperner1} Let $X$ be an $n$-dimensional psm. Then for every $(n+1)$-coloring $\varphi$ of
$X$, the number of $\varphi$-rainbow facets of $X$  is even.
\end{lemma}

\begin{proof}
Consider all pairs $A \supset B$ with $A$ is a facet of $X$, $B$ being $(n-1)$-dimensional and $\varphi$-rainbow, where
on the vertices of $B$, $\varphi$ takes all values except for $n+1$.
We count the numberof such pairs in two different ways.

Each $(n-1)$-simplex $B$ that $\varphi$ maps onto $\{1,\ldots,n\}$ participates
in exactly two such pairs, once with each of the two facets that contain it. Hence the
total count is even.

Each rainbow facet $A$ participates in exactly one good pair $A \supset B$, where $B= A \setminus \varphi^{-1}(n+1)$.\\
The claim follows.
\end{proof}

Thus, in particular, if we 3-color the vertices of a
triangulated planar graph, so that the outer face is rainbow, then there
must be at least one more rainbow face in the triangulated planar
graph.

We say that an $n$-dimensional psm is {\em colorable} if it has a $(n+1)$-coloring for which  no edge is monochromatic. In other words, an $(n+1)$-coloring for which all facets are rainbow.

\begin{lemma}
\label{cor:sperner2} Let $\delta$ be a $2$-coloring of a colorable psm $X$.
Then the number of $\delta$-monochromatic facets of $X$ is even.
\end{lemma}

\begin{proof}
By assumption $X$ is colorable, so let $\chi$ be some proper $(n+1)$-coloring of $X$. Define next a new $(n+1)$-coloring $\varphi$ via $\varphi := \chi + \delta \mod (n+1)$. By assumption, every facet is $\chi$-rainbow and the addition ($\bmod (n+1)$) of a constant value of a monochromatic $\delta$ does not change this property. In other words, every $\delta$-monochromatic  facet is $\varphi$-rainbow. We claim the reverse implication holds as well. Indeed, if $\delta$ is not constant on the facet $A$, then we can find two vertices $x, y \in A$ for with $\delta(x)=1, \delta(y)=2$ and $\chi(y)=\chi(x)+1 \mod (n+1)$. But then no vertex $z \in A$ satisfies $\varphi(z) = \chi(x)+2 \mod (n+1)$. In other words, a facet is $\varphi$-rainbow, iff it is $\delta$-monochromatic.
By Lemma~\ref{lem:sperner1}, the proof is complete.
\end{proof}

\subsection{MC as a pseudomanifold}
\label{sec:appLBproof}

Here we prove Theorem~\ref{thm:mainLB} by using psm's
and Lemma~\ref{cor:sperner2}. While it is true that  psm's can be realized geometrically, we do not refer
to such realizations. Still, as mentioned above, planar triangulations can be useful in guiding one's intuition in this area.

Given the $N$ words, our plan is to to construct a psm $X$ that encodes certain possible executions of the MC algorithm. Vertices of $X$ correspond to states of individual players, and facets correspond to reachable configurations.  Since we limit ourselves to schedules that involve only $n$ out of the $N$ available players, every facet has $n$ vertices, so that $\dim X = n-1$.

In the setting of Theorem~\ref{thm:mainLB} the scheduler selects one player from each of
the $n-2$ pairs (and adds in players $P_1, P_2$). This gives $2^{n-2}$ possible initial
configurations, which we call {\em initial facets}. Note, however, that these $2^{n-2}$ sets do not constitute the collection of facets of a psm's. An $(n-2)$-dimensional face that contains one player from each of the $n-2$ pairs and exactly one of $P_1, P_2$ is covered by exactly one initial facet, in violation of condition (\ref{pseudom}).
To overcome this difficulty we add two auxiliary vertices called $A_1$ and $A_2$, where $A_1$ is viewed as being paired with $P_1$,
and $A_2$ with $P_2$. This yields the $2^n$ sets that are obtained by making all
possible choices of one vertex from each of the $n$ pairs. It is easily verified that this collection constitutes the set of facets of an $(n-1)$-dimensional psm. These $2^n$ facets include the $2^{n-2}$ initial
configurations, and $2^n - 2^{n-2}$ {\em auxiliary facets}, those that include at least one of $A_1, A_2$. Figure~\ref{fig:feige1} illustrates the situation for $n=3$. There are six vertices, which correspond to $N = 2n-2= 4$ players plus two auxiliary players. The vertices are 3-colored, where each pair of players (say $P_1$ and $A_1$) are equally colored. The planar graph has eight faces (including the outer face), which are the $2^3=8$ facets.

\begin{figure}[htb]
\includegraphics{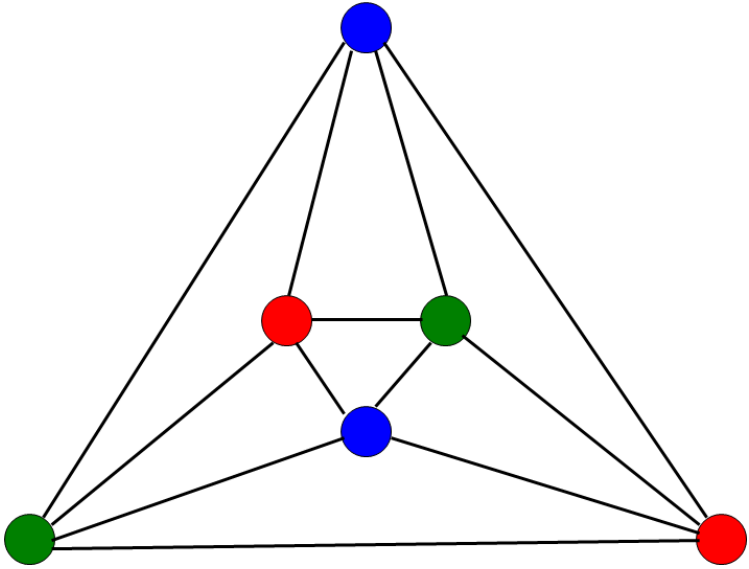}
\caption{A 3-colorable cell system of rank 3} \label{fig:feige1}
\end{figure}

Let us now introduce a 2-coloring of the vertices. We partition the $2n-2$ chairs into two subsets of cardinality $n-1$ each, called the 0-chairs and the 1-chairs. The initial chair of $P_1, P_2$ is a 1-chair whereas the initial chair of $A_1, A_2$ is a 0-chair. Also, within each pair of players (out of the $n-2$ original pairs), one
starts at a 0-chair and the other at a 1-chair. These requirements are consistent, since by assumption the only coincidence among initial chairs is that $\pi_1[1]=\pi_2[1]$.

\begin{proposition} \label{pro:initial}
The collection of all subsets of the $2^n$ initial sets is a colorable $(n-1)$-dimensional psm. In the above-described 2-coloring there is exactly one monochromatic auxiliary cell.
\end{proposition}

\begin{proof}
We already noticed that this is indeed a psm. Let us associate a unique color to each pair of paired vertices. This makes every facet
rainbow, so the psm is indeed colorable.

In the above-described 2-coloring there is indeed a unique monochromatic auxiliary facet. This the facet that contains $A_1$, $A_2$, and the
$n-2$ players (one from each pair) who start from a 0-chair.
\end{proof}

Starting from the initial system, the rules of MC allow the
scheduler to generate new psm's whose facets represent
reachable configurations. We remark that unlike the initial
system, it may happen that several facets correspond to the same configuration. This fact will cause no harm to us.

Let us consider a move of the scheduler relative to some pseudomanifold
$PSM$. At a given configuration corresponding to a facet in $PSM$, if
some players are in conflict, the scheduler may select two such
players that occupy the same chair, and move either one of them or both. Hence given the two players and their states
(say, corresponding to vertices $v_1$ and $v_2$ in $PSM$), two new states are exposed by this choice of three
possible moves. These correspond to two new vertices (say $v_1'$
and $v_2'$) in a new psm. The given configuration can
be moved to one of three new configurations, which in our
psm representation corresponds to splitting the facet $\sigma$ that
corresponds to the given configuration to three new facets.
These new facets are obtained as
$v_1$ or $v_2$ or both in $\sigma$ are replaced by $v_1'$,
$v_2'$ or both, respectively. However, we are not done yet. The edge $\{v_1, v_2\}$ may be contained in several facets of $PSM$. Each such facet corresponds to
a configuration in which the scheduler can apply any of the same
three types of moves (moving $v_1$ to $v_1'$, moving $v_2$ to
$v_2'$, or moving both). Hence every such facet is split in
three as described above. This completes the description of the new pseudomanifold $PSM'$.

We say that the above process {\em subdivides} the edge $\{v_1,v_2\}$. Figure~\ref{fig:feige2} illustrates the
subdivision process when $n = 3$. (It is convenient to have $A_1$
and $A_2$ on the outer faces of such drawings, so
that edges correspond to straight line segments.)

\begin{figure}[htb]
\includegraphics{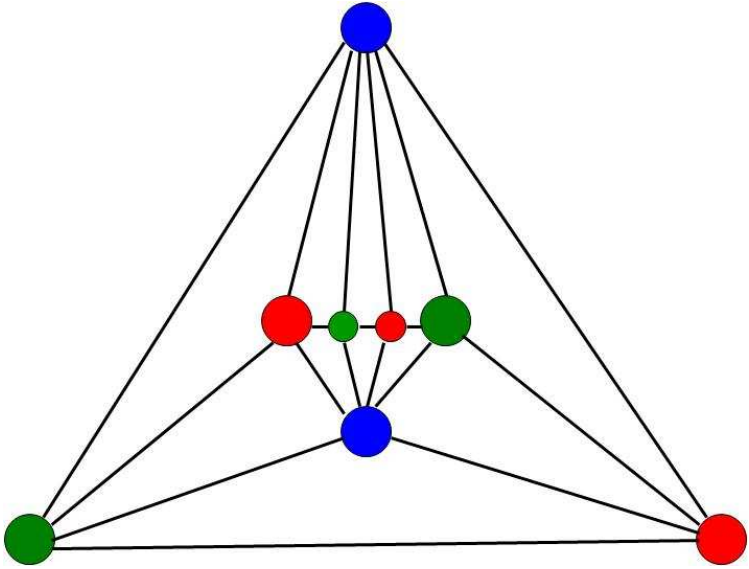}
\caption{The cell system when $n=3$ after one step by the
scheduler. The auxiliary players are the red and green vertices on
the outer face.} \label{fig:feige2}
\end{figure}

\begin{proposition}
\label{pro:auxiliary} No move of the scheduler can subdivide
an auxiliary cell.
\end{proposition}

\begin{proof}
Observe that $A_1$ and $A_2$ are never involved in a
subdivided edge (they are not true players and the scheduler
cannot move them). As to the rest of the vertices in an auxiliary
cell, they all correspond to different chairs, and hence cannot
pair up to create a subdivided edge.
\end{proof}

We can extend the 2-coloring of $PSM$ to that of $PSM'$ by giving the
two new vertices $v_1'$ and $v_2'$ the 0/1 color of the chairs
corresponding to their respective states.

\begin{proposition}
\label{pro:CS'} The simplicial complex $PSM'$ described above is a colorable pseudomanifold. The 2-coloring
described for it has exactly one monochromatic auxiliary cell.
\end{proposition}

\begin{proof}
By induction and using Proposition~\ref{pro:initial} as the base
case, we may assume that $PSM$ that gives rise to $PSM'$ is a
psm. We need to show that edge subdivision (and the implied
cell subdivisions) maintains the property that each $(n-1)$-face is
covered by exactly two facets. This amounts to a simple case
analysis, presented for completeness in
Section~\ref{sec:caseanalysis}.

The colorability property of $PSM'$ is inherited from $PSM$, because
every new facet contains the same set of players as its parent facet
that was subdivided.

The 2-coloring property is a direct consequence of
Proposition~\ref{pro:initial} and Proposition~\ref{pro:auxiliary}.
\end{proof}

We now reach a crucial part of our proof. Consider a psm
$PSM$ generated by the process described above (starting from the
initial psm and subdividing edges and their respective
cells). By Proposition~\ref{pro:CS'} $PSM$ is colorable. Hence
for the 2-coloring that we associate with it, there are an even
number of monochromatic facets, by Lemma~\ref{cor:sperner2}. As
exactly one of these facets is auxiliary (by
Proposition~\ref{pro:CS'}), there is at least one non-auxiliary
monochromatic facet, and this facet corresponds to a reachable
configuration. Since there are only $n-1$ 0-chairs and $n-1$ 1-chairs,
by pigeonhole there are at least two players on the same chair in this configuration. Therefore the scheduler can
subdivide the respective edge. It follows that the scheduler can
continue to subdivide the psm indefinitely, creating arbitrarily large psm's.
Together with the next proposition this completes
the proof of Theorem~\ref{thm:mainLB}.

\begin{proposition}
\label{pro:tree} If the pseudomanifold
that is generated by the subdivision process has $2^n 3^{t-1}+1$
facets, then there is a scheduler strategy that makes the game last for at least $t$ steps.
\end{proposition}

\begin{proof}
We describe an iterative process in which we
construct a rooted tree $T$ the reflects the process of generating the psm. The $2^n$ facets of the initial psm are $2^n$
children of the root. To represent a step in which we subdivide a facet $\sigma$, we let the leaf corresponding to $\sigma$ have three
children. Thus every
internal (neither root nor leaf) node of $T$ has exactly three children. The leaves of $T$ correspond to facets, of which $2^n - 2^{n-2}$ are
auxiliary, and the others corresponding to
reachable configurations. (These need not be distinct reachable
configurations, and not every reachable configuration is necessarily
represented by the subdivision process). If $T$ has $2^n
3^{t-1}+1$ leaves, then there must be a leaf at depth
at least $t+1$. The root to leaf path describes a possible schedule of $t$ moves, as claimed.
\end{proof}

\subsection{A case analysis} \label{sec:caseanalysis}

Here we present the case analysis used for the proof of
Proposition~\ref{pro:CS'}.

We first consider various options for an $(n-1)$-face $F$ of $PSM'$.

\begin{enumerate}

\item If $F$ contains none of $v_1,v_2,v_1',v_2'$, then it
must be a face in $PSM$ as well. The facets that contain $F$ were
not subdivided (since they could contain at most one of $v_1$ and
$v_2$), and hence also in $CS'$ there are exactly two facets containing $F$.

\item If $v_1 \in F$ the $F$ must be an $(n-1)$-face in $PSM$ as well.
Possibly some facet $C$ in $PSM$ contains $F$ and was subdivided in
going from $PSM$ to $PSM'$), if $C$ contained $v_2$. But then $PSM$
has the facet $C'$ with $v_2'$ replacing $v_2$, and the number of
facets in $PSM'$ that contain $F$ is the same as their number in $PSM$.

\item If $v_2 \in F$ of $CS'$ the same argument applies.

\item Note that $v_1, v_2 \in F$ is impossible, since edge $\{v_1,v_2\}$ was subdivided).
\end{enumerate}

We next turn to analyze various options for an $(n-1)$-face $F'$ of $PSM'$
that is not a face in $PSM$, where, say $v_1' \in F'$. Indeed $F'$ is not a face in $PSM$ (where $v_1'$ does not exist), and it results from a
subdivision step. We need to show that it is covered exactly twice following the subdivision. Consider two subcases.

\begin{itemize}
\item $F'$ contains neither $v_2$ nor $v_2'$. Then
one facet containing $F'$ has $v_2$ as its remaining vertex, and the
other cell has $v_2'$.

\item $F'$ contains either $v_2$ or $v_2'$ (note that it cannot contain both).
Consider the unique $(n-1)$-face $F$ of $PSM$ that is derived from $F'$ by
replacing $v_1'$ by $v_1$, and also $v_2'$ by $v_2$ if $v_2' \in
F'$ (this last replacement is not required if $v_2 \in F$). Then $F$
is covered twice in $PSM$, say by $C_1$ and $C_2$. Each of them is
subdivided (because they contain both $v_1$ and $v_2$), giving
rise to two facets in $PSM'$ that contain $F'$. No other facet in $PSM'$ can
contain $F'$.

\end{itemize}

The case where $v_2' \in F'$ is handled similarly.

\section{Oblivious MC algorithms via the probabilistic method}
\label{sec:efficientMC}

We start with an observation that puts Theorems~\ref{thm:7n}
and~\ref{thm:permutations} (as well as Theorem~\ref{thm:mainfull})
in an interesting perspective. The expected number of pairwise
conflicts in a random configuration is exactly ${n \choose 2}/m$.
In particular, when $m \gg n^2$, most configurations are {\em
safe} (namely, have no conflicts). Therefore, it in not
surprising that in this range of parameters $n$ random words would
yield an oblivious $MC(n,m)$ algorithm. However, when $m = O(n)$,
only an exponentially small fraction of configurations are safe,
and the existence of oblivious $MC(n,m)$ algorithms is far from
obvious.

\subsection{Full words with $O(n)$ chairs, allowing repetitions}

Theorem~\ref{thm:7n} can be viewed as a (non constructive)
derandomization of the randomized MC algorithm in which players
choose their next chair at random (and future random decisions of
players are not accessible to the scheduler). Standard techniques
for derandomizing random processes involve taking a union bound
over all possible bad events, which in our case corresponds to a
union bound over all possible schedules. The immediate
scheduler has too many options (and so does the pairwise immediate
scheduler), making it infeasible to apply a union bound. For this reason, we
shall consider in this section the canonical scheduler, which is
just as powerful (see
Section~2). In every unsafe configuration, the canonical
scheduler has just three possible moves to choose from. This
allows us to use a union bound. We now prove
Theorem~\ref{thm:7n}.

\begin{proof}
Each of the $N$ words is chosen independently at random as a
sequence of $L$ chairs, where each chair in the sequence is chosen
independently at random. We show that with high probability
(probability tending to~1 as the constant $c$ grows),
this choice satisfies Theorem~\ref{thm:7n}.

It is easy to verify that in this random construction, with high
probability, all words are full. To see this note that the
probability that chair $j$ is missing from such a random word is
$((m-1)/m)^L$. Consequently, the probability that a word chosen
this way is not full is $\le m((m-1)/m)^L$. Therefore, the
expected number of non-full words is $\le m\cdot N
\cdot((m-1)/m)^L$. But with our choice of parameters $m = 7n$ and
$L = cn\log N$, we see that $m\cdot N \cdot((m-1)/m)^L=o(1)$,
provided that $c$ is large enough.

In our approach to the proof we keep track of all possible
schedules. To this end we use ``a logbook'' that is the complete
ternary tree $\cal T$ of depth $L$ rooted at $r$. Associated with
every node $v$ of $\cal T$ is a random variable $X_v$. The values
taken by $X_v$ are system configurations. For a given choice of
words and an initial system configuration we define the value of
$X_r$ to be the chosen initial configuration. Every node $v$ has
three children corresponding to the three possible next
configurations that are available to the canonical scheduler at
configuration $X_v$.

Another important ingredient of the proof is a {\em potential}
function (defined below) that maps system configurations to the
nonnegative reals. It is also convenient to define an (artificial)
``empty'' configuration of $0$ potential. Every safe configuration
has potential $1$, and every non-empty unsafe configuration has
potential $>10$. If the node $u$ is a descendant of $v$ and the
system configuration $X_v$ is safe, then we define $X_u$ to be the
empty configuration.

We thus also associate with every node of $\cal T$ a nonnegative
random variable $P=P_v$ that is the potential of the (random)
configuration $X_v$. The main step of the proof is to show that if
$v_1,v_2,v_3$ are the three children of $v$, then $\sum_{i=1}^3
\mathbb{E}(P_{v_i}) \le r \mathbb{E}(P_v)$ for some constant $r \le
0.99$. (Note that this inequality holds as well if $X_v$ is either
safe or empty). This exponential drop implies that
\[
\mathbb{E}(\sum_{v\mbox{~is~a~leaf~of~}{\cal
T}}(P_v))=\sum_{v\mbox{~is~a~leaf~of~}{\cal T}}\mathbb{E}(P_v) =
o(1)
\]
provided that $L$ is large enough. This implies that with
probability $1-o(1)$ (over the choice of random words) all leaves
of $\cal T$ correspond to an empty configuration. In other words
every schedule terminates in fewer than $L$ steps.

We turn to the details of the proof. A configuration with $i$
occupied chairs is defined to have potential $x^{n-i}$, where $x >
1$ is a constant to be chosen later. In a nonempty configuration
the potential can vary between $1$ and $x^{n-1}$, and it equals
$1$ iff the configuration is safe.

Consider a configuration of potential $x^{n-i}$ (with $i < n$),
where the canonical pair is $(\alpha,\beta)$. It has three
children representing the move of either $\alpha$ or $\beta$ or
both. Let us denote $\rho=i/m$ and $\rho' = (i-1)/m$. When a single
player moves, the number of occupied chairs can stay unchanged,
which happens with probability $\rho$. With probability $1-\rho$
one more chair will be occupied and the potential gets divided by
$x$. Consider next what happens when both players move. Here the possible outcomes
(in terms of number of occupied chairs) depend on whether
there is an additional player $\gamma$
currently co-occupying the same chair as $\alpha$ and $\beta$.
It suffices to perform the analysis in the less favorable case in which
there is no such player $\gamma$, as this provides an upper bound on the potential also
for the case that there is such a player.
With probability $(\rho')^2$ both $\alpha$ and $\beta$ move to occupied chairs and
the potential gets multiplied by $x$. With probability
$\rho'(1-\rho') + (1 - \rho')\rho = (\rho + \rho')(1 - \rho')$ the number of occupied chairs (and hence the
potential) does not change. With probability $(1 - \rho')(1-\rho)$ the
number of occupied chairs grows by one and the potential gets
divided by $x$.

It follows that if $v$ is a node of $\cal T$ with children
$v_1,v_2,v_3$ and if the configuration $X_v$ is unsafe and
nonempty then $\sum_{i=1}^3 \mathbb{E}(P_{v_i}) \le
\mathbb{E}(P_v)(2\rho + 2(1-\rho)/x + (\rho')^2x + (\rho + \rho')(1-\rho') +
(1-\rho)(1 - \rho')/x)$. Recall that $x > 1$ and $\rho' < \rho < 1$. This implies that the last expression increases if $\rho'$ is replaced by $\rho$, and thereafter it is maximized when $\rho$ attains
its largest possible value $q = (n-1)/m$. We conclude that

\[
\sum_1^3 \mathbb{E}(P_{v_i}) \le \mathbb{E}(P)(2q + 2(1-q)/x +
q^2x + 2q(1-q) + (1-q)^2/x).
\]
We can choose $q = 1/7$ and $x = 23/2$ to obtain $\sum_{i=1}^3
\mathbb{E}(P_{v_i}) \le r \mathbb{E}(P_v)$ for $r<0.99$. This
guarantees an exponential decrease in the expected sum of
potentials and hence termination, as we now explain.

It follows that for every initial configuration the expected sum
of potentials of all leaves at depth $L$ does not exceed $x^{n-1}$
(the largest possible potential) times $r^L$. On the other hand,
if there is at least one leaf $v$ for which the configuration
$X_v$ is neither safe nor empty, then the sum of potentials at
depth $L$ is at least $x > 1$. Our aim is to show that with high
probability (over the choice of $N$ words), all runs have length
$<L$: (i) For every choice of $n$ out of the $N$ words, (ii) Each
selection of an initial configuration, and (iii) Every canonical
scheduler's strategy. The $n$ words can be chosen in $N\choose n$
ways. For every $n$ words, there are $L^n$ possible initial
configurations. The probability of length-$L$ run from a given
configuration is at most $x^{n-1} r^L$, where $x = 23/2$ and $r <
0.99$. Therefore our claim is proved if ${N \choose n} \cdot
x^{n-1} r^L \le o(1)$. This inequality clearly holds if we let $L
= cn\log N$ with $c$ a sufficiently large constant. This completes
the proof of Theorem~\ref{thm:7n}.

\end{proof}

A careful analysis of the proof of Theorem~\ref{thm:7n} shows that
it actually works as long as $\frac mn > 4+2\sqrt{2} = 6.828..$.
It would be interesting to determine the value of $\liminf_{n
\rightarrow \infty} \frac mn$ for which $n$ long enough random
words over an $m$-letter alphabet constitute, with high
probability, an oblivious $MC(n,m)$ algorithm.

\subsection{Permutations over $O(n)$ chairs}
\label{ap:permutations}

The argument we used to prove Theorem~\ref{thm:7n} is
inappropriate for the proof of Theorem~\ref{thm:permutations}.
Theorem~\ref{thm:permutations} deals with random permutations,
whereas in the proof of Theorem~\ref{thm:7n} we use words of
length $\Omega(n\log n)$. (Longer words are crucial there for two
main reasons: To guarantee that words are full and to avoid
wrap-around. The latter property is needed to guarantee
independence.) Indeed in proving Theorem~\ref{thm:permutations}
our arguments are substantially different. In particular, we work
with a pairwise immediate scheduler, and unlike the proof of
Theorem~\ref{thm:7n}, there does not appear to be any significant
benefit (e.g., no significant reduction in the ratio $\frac mn$)
if a canonical scheduler is used instead.

We first prove the special case $N=n$ of Theorem
\ref{thm:permutations}.

\begin{thm}
\label{thm:main} If $m \geq cn$ where $c > 0$ is a sufficiently
large constant, then there is a family of $n$ permutations on
$[m]$ which constitute an oblivious $MC(n,m)$ algorithm.
\end{thm}

We actually show that with high probability, a set of random
permutations $\pi_1, \ldots, \pi_n$ has the property that in every
possible schedule the players visit at most $L=O(m \log m)$
chairs. Our analysis uses the approach of deferring random
decisions until they are actually needed. For each of the $m^n$
possible initial configuration, we consider all possible sequences
of $L$ locations. For each such sequence we fill in the chairs in
the locations in the sequence at random, and prove that the
probability that this sequence represents a possible schedule is
extremely small -- so small that even if we take a union bound
over all initial configurations and over all sequences of length
$L$, we are left with a probability much smaller than~1.

The main difficulty in the proof is that since $L \gg m$, some
players may completely traverse their permutation (even more than
once) and therefore the chairs in these locations are no longer
random. To address this, we partition the sequence of moves into
$L/t$ blocks, where in each block players visit a total of $t$
locations. We can and will assume that $t$ divides $L$. We take
$t=\delta m$ for some sufficiently small constant $\delta$, and $n
=\epsilon m$, where $\epsilon$ is a constant much smaller than
$\delta$. This choice of parameters implies that within a block,
chairs are essentially random and independent. To deal with
dependencies among different blocks, we classify players (and
their corresponding permutations) as {\em light} or {\em heavy}. A
player is {\em light} if during the whole schedule (of length $L$)
it visits at most $t/\log m=o(t)$ locations. A player that visits
more than $t/\log m$ locations during the whole sequence is {\em
heavy}. Observe that for light players, the probability of
encountering a particular chair in some given location is at most
$\frac{1}{m - o(t)} \le \frac{1+o(1)}{m}$. Hence, the chairs
encountered by light players are essentially random and
independent (up to negligible error terms). Thus it is the heavy
players that introduce dependencies among blocks. Every heavy
player visits at least $t/\log m$ locations, so that $n_h$, the
number of heavy players does not exceed $n_h \le (L \log
m)/t=O(\log^2 m)$. The fact that the number of heavy players is
small is used in our proof to limit the dependencies among blocks.

The following lemma is used to show that in every block of length
$t$ the number of locations that are visited by heavy players
is not too large. Consequently,
sufficiently many locations are visited by
light players. In the lemma we use the following notation. A
segment of $k$ locations in a permutation is said to have {\em
volume} $k-1$. Given a collection of locations, a chair is {\em
unique} if it appears exactly once in these locations.

\begin{lemma}
\label{lem:heavy} Let $n_h \leq m/\log^2 m$ and let
$\delta > 0$ be a sufficiently small constant. Consider $n$ random
permutations over $[m]$. Select any $n_h$ of the permutations and a
starting location in each of them. Choose next intervals in the
selected permutations with total volume $t'$ for some $t/10 \le t'
\le t$. With probability $1-o(1)$ for every such set
of choices at least $4t'/5$ of the chairs in the chosen intervals
are unique.
\end{lemma}

\begin{proof}
We first note that we will be using the lemma with $n_h = O(\log^2
n)$. Also, if a list of letters contains $u$ unique letters (i.e.,
they appear exactly once) and $r$ repeated letter (i.e., appearing
at least twice), then it has $d=u+r$ distinct letters and length
$\lambda \ge u+2r$. In particular $d \le (\lambda+u)/2$.

There are ${n \choose n_h}$ ways of choosing $n_h$ of the
permutations. Then, there are $m^{n_h}$ choices for the initial
configuration. We denote by $s_i$ the volume of the  $i$-th
interval, so that $\sum_{i=1}^{n_h} s_i =t'$. Therefore there are
${t' + n_h - 1 \choose n_h-1} \le m^{n_h}$ ways of choosing the
intervals with total volume $t'$. Since the volume of every
interval is at most $t'$ we have that the probability that a
particular chair resides at a particular location in this interval
is at most $1/(m-t')$. This is because the permutation is random
and at most $t'$ chairs appeared so far in this interval.
Therefore the probability that a sequence of $t'$ labels involves
less than $0.9t'$ distinct chairs is at most
\begin{eqnarray*}
{m \choose 0.9t'} \left(\frac{0.9t'}{m-t'}\right)^{t'} &\leq&
\left(\frac{em}{0.9t'}\right)^{0.9t'}
\left(\frac{0.9t'}{m-t'}\right)^{t'} \leq
e^{t'} \left(\frac{m}{m-t'}\right)^{0.9t'} \left(\frac{t'}{m-t'}\right)^{0.1t'}\\
&\leq& 4^{t'} (2\delta)^{0.1t'} \ll e^{-t'}.
\end{eqnarray*}
Explanation: The set of chairs that appear in these intervals can
be chosen in ${m \choose 0.9t'}$ ways. The probability that a
particular location in this union of intervals is assigned to a
chair from the chosen set does not exceed $\frac{0.9t'}{m-t'}$. In
addition $m/(m-t') \leq (1+\delta)$, $t'/(m-t') \leq 2\delta$ and
$\delta$ is a very small constant.

Now we take a union bound over all choices of $n_h$ permutations, all
starting locations and all collection of intervals with total
volume $t'$. It follows that the probability that there is a
choice of intervals of volume $t'$ that span $\le n_h$
permutations and contain fewer than $9t'/10$ distinct chairs
is at most
\[
m^{3n_h}e^{-t'} = o(1).
\]
In the above notation $\lambda=t'$ and $d \ge 0.9t'$ which yields
$u \ge 0.8t'$ as claimed.
\end{proof}

Since the conclusion of this lemma holds with probability $1-o(1)$
we can assume that our set of permutations satisfies it. In
particular, in every collection of intervals in these permutations
with total volume $\frac{t}{10} \le t' \le t$ that reside in
$O(\log^2 m)$ permutations there are at least $4t'/5$ unique
chairs.

As already mentioned, we break the sequence of $L$ locations
visited by players into blocks of $t$ locations each. We analyze
the possible runs by considering first the {\em breakpoints
profile}, namely where each block starts and ends on each of the
$n$ words. There are $m^n$ possible choices for the starting
locations. If, in a particular block player $i$ visits $s_i$
chairs, then $\sum_{i=1}^n s_i=t$. Consequently the parameters
$s_1,\ldots,s_n$ can be chosen in ${t+n-1 \choose n} \le 2^{t+n}$
ways. There are $L/t$ blocks, so that the total number of possible
breakpoints profiles is at most $m^n (2^{t+n})^{L/t} \le m^n
2^{2L}$ (here we used the fact that $t > n$). Clearly, by
observing the breakpoints profile we can tell which players are
light and which are heavy. We recall that there are at most
$O(\log^2 m)$ heavy players, and that the premise of Lemma
\ref{lem:heavy} can be assumed to hold.

Let us fix an arbitrary particular breakpoints profile $\beta$. We
wish to estimate the probability (over the random choice of
chairs) that some legal sequence of moves by the pairwise
immediate scheduler yields this breakpoints profile $\beta$. Let
$B$ be an arbitrary block in $\beta$. Let $p(B)$ denote the
probability over choice of random chairs and {\em conditioned over
contents of all previous blocks in $\beta$} that there is a legal
sequence of moves by the pairwise immediate scheduler that
produces this block $B$.

\begin{lemma}
\label{lem:pB}
For $p(B)$ as defined above we have that $p(B) \le 8^{-t}$.
\end{lemma}

\begin{proof}
The total number of chairs encountered in block $B$ is $n \ll t$
(for the initial locations) plus $t$ (for the moves). Recall that
the set of heavy players is determined by the block-sequence
$\beta$. Hence within block $B$ it is clear which are the heavy
players and which are the light players. Let $t_h$ (resp.
$t_{\ell} = t - t_h$) be the number of chairs visited by heavy
(resp. light) players in this block.  The proof now breaks into
two cases, depending on the value of $t_h$.

{\bf Case 1: $t_h \leq 0.1t$.} Light players altogether visit $n + t_{\ell}$
chairs ($n$ initial locations plus $t_{\ell}$ moves). If $u$ of
these chair are unique, then they visit at most $(n +
t_{\ell}+u)/2$ distinct chairs. But a chair in this collection
that is unique is either: (i) One of the $n$ chairs where a player
terminates his walk, or, (ii) A chair that a light player
traverses due to a conflict with a heavy player, and there are at
most $t_h$ of those. Consequently, the number of distinct chairs
visited by light players does not exceed $(n + t_{\ell}+n+t_h)/2 =
t/2 + n$.

Fix the set $S$ of $t/2+n$ distinct chairs that we are allowed to
use. There are ${m \choose n+t/2}$ choices for $S$. Now assign
chairs to the locations one by one, in an arbitrary order. Each
location has probability of at most $(1+o(1))\frac{n+t/2}{\MM}$ of
receiving a chair in $S$. Since we are dealing here with light
players, we have exposed only $o(m)$ chairs for each of them (in
$B$ and in previous blocks of $\beta$), and as mentioned above,
this can increase the probability by no more that a $1+o(1)$
factor.

Hence the probability that the segments traversed by the light
players contain only $n+t/2$ chairs is at most

\begin{eqnarray*}
 & {\MM \choose n+t/2}\left((1+o(1))\frac{n+t/2}{\MM}\right)^{t_{\ell}}
\leq  \left(\frac{em}{n+t/2}\right)^{n+t/2} 2^{t_{\ell}} \left(\frac{n+t/2}{m}\right)^{t_{\ell}} \\
&\leq (2e)^t \left(\frac{n+t/2}{m}\right)^{(t_{\ell}-t_h)/2-n}
\leq (2e)^t \big(t/m\big)^{t/4} < 8^{-t}.
\end{eqnarray*}
Here we used that $t_h+t_{\ell}= t$, $t_h\leq 0.1t$, $t_l \geq
0.9t$ and $n \ll t \ll m$.

{\bf Case 2: $t_h \geq 0.1t$.} Let us
reveal first the chairs visited by the heavy players. By
Lemma~\ref{lem:heavy}, we find there at least $4t_h/5$ unique
chairs. In order that the heavy players traverse these chairs,
they must be visited by light players as well. Hence the
$t_{\ell}$ locations visited by light players must include all
these $0.8t_h$ pre-specified chairs. We bound the probability of
this as follows. First choose for each of the $0.8t_h$
pre-specified chairs a particular location where it should appear
in the intervals of light players. The number of such choices is
$\le t_{\ell}^{0.8t_h}$. As mentioned above the probability that a
particular chair is assigned to some specific location is
$(1+o(1))/m$. Therefore the probability that $0.8t_h$
pre-specified chairs appear in the light intervals is at most
$t_{\ell}^{0.8t_h}\big((1+o(1))/m\big)^{0.8t_h}$. Thus the
probability that a schedule satisfying the condition of the lemma
exists is at most
\begin{eqnarray*}
 t_{\ell}^{0.8t_h}\big((1+o(1))/m\big)^{0.8t_h} \leq  \big(2t/m\big)^{0.8t_h}
&\leq&  \big(2t/m\big)^{t/15} < 8^{-t},
\end{eqnarray*}
where we used that $n \ll t \ll m$.
\end{proof}

Lemma~\ref{lem:pB} implies an upper bound of $p(B)^{L/t} = 8^{-L}$
on the probability there is a legal sequence of moves by the
pairwise immediate scheduler that gives rise to breakpoints
profile $\beta$. Taking a union bound over all block sequences
(whose number is at most $m^n 2^{2L} \le 6^L$, by our choice of
$L=C m\log m$ for a sufficiently large constant $C$),
Theorem~\ref{thm:main} is proved.

Observe that the proof of Theorem~\ref{thm:main} easily extends to
the case that there are $N = m^{O(1)}$ random permutations out of
which one chooses $n$. We simply need to multiply the number of
possibilities by $N^n$, a term that can be absorbed by increasing
$m$, similar to the way the term $\MM^n$ is absorbed. In
Lemma~\ref{lem:heavy} we need to replace ${n \choose n_h}$ by ${N
\choose n_h}$, and the proof goes through without any change
(because $n_h$ is so small). This proves
Theorem~\ref{thm:permutations}.

\subsection{Explicit construction with permutations and $m=O(n^2)$}

In this section we present for every integer $d \ge 1$ an explicit
collection of $n^d$ permutations on $m=O(d^2 n^2)$ such that every
$n$ of these permutations constitute an oblivious $MC(n,m)$
algorithm. This proves Theorem~\ref{thm:explicitperm}.

We let $LCS(\pi,\sigma)$ stand for the length of the longest
common subsequence of the two permutations $\pi$ and $\sigma$,
considered cyclically. (That is, we may rotate $\pi$ and $\sigma$
arbitrarily to maximize the length of the resulting longest common
subsequence). The following easy claim is useful.

\begin{proposition}
\label{pairwise} Let $\pi_1, \ldots, \pi_n$ be permutations of
$\{1, \ldots, \MM\}$ such that $LCS(\pi_i,\pi_j) \le r$ for all $i
\neq j$. If $m > (n-1)r$, then in every schedule none of the
$\pi_i$ is fully traversed.
\end{proposition}

\begin{proof}
By contradiction. Consider a schedule in which one of the
permutations is fully traversed, say that $\pi_1$ is the first
permutation to be fully traversed. Each move along $\pi_1$
reflects a conflict with some other permutation. Hence there is a
permutation $\pi_i, i>1$ that has at least $\MM/(n-1)$ agreements
with $\pi_1$. Consequently, $r \ge LCS(\pi_1,\pi_i) \geq
\frac{\MM}{(n-1)}$, a contradiction.
\end{proof}

This yields an inexplicit oblivious $MC(n,m)$ algorithm with
$m=O(n^2)$, since (even exponentially) large families of
permutations in $[m]$ exist where every two permutations have an
LCS of only $O(\sqrt m)$. We omit the easy details. On the other
hand, we should notice that by~\cite{BeameBN09} this approach is
inherently limited and can, at best yield bounds of the form $m
\le O(n^{3/2})$.

We now present an explicit construction that uses some algebra.

\begin{lemma}
\label{explicit} Let $p$ be a prime power, let $d$ be a positive
integer and let $\MM=p^2$. Then there is an explicit family of
$(1-o(1))\MM^d$ permutations of an $\MM$-element set, where the
LCS of every two permutations is at most $4d\sqrt{\MM}$.
\end{lemma}

\begin{proof}
Let $\mathbb{F}$ be the finite field of order $p$. Let ${\cal M}
:= \mathbb{F}\times \mathbb{F}$, and $\MM=p^2=|{\cal M}|$. Let $f$
be a polynomial of degree $2d$ over $\mathbb{F}$ with vanishing
constant term, and let $j \in \mathbb{F}$. We call the set
$B_{f,j}=\{(x,f(x)+j) | x \in \mathbb{F}\}$ {\em a block}. We
associate with $f$ the following permutation $\pi_f$ of ${\cal
M}$: It starts with an arbitrary ordering of the elements in
$B_{f,0}$ followed by $B_{f,1}$ arbitrarily ordered, then of
$B_{f,2}$ etc. A polynomial of degree $r$ over a field has at most
$r$ roots. It follows that for every two polynomials $f \neq g$ as
above and any $i, j \in \mathbb{F}$, the blocks $B_{f,i}$ and
$B_{g,j}$ have at most $2d$ elements in common. There are
$(p-1)\cdot p^{2d-1}=(1-o(1))\MM^d$ such polynomials.  There are
$p$ blocks in $\pi_f$ and in $\pi_g$, so that $LCS(\pi_f,\pi_g)
\le 4dp$, as claimed.

\end{proof}

\section{Discussion and Open Problems}

This work originated with the introduction of the concept of oblivious distributed algorithms. In the present paper we concentrated on oblivious MC
algorithms, a topic which yields a number of interesting mathematical challenges. We showed that $m \ge 2n-1$ chairs are necessary and
sufficient for the existence of an oblivious $MC$ algorithm with
$n$ processors. Still, our construction involves very long
words. It is interesting to find explicit constructions with $m =
2n-1$ chairs and substantially shorter words.

In other ranges of the problem we can show, using the
probabilistic method, that oblivious $MC(n,m)$ algorithms exist
with $m = O(n)$ and relatively short full words. We still do not
have explicit constructions with comparable properties. We would also like
to determine $\liminf \frac{m}{n}$ such that $n$ random words over
an $m$ letter alphabet typically constitute an oblivious $MC(n,m)$
algorithm.

Computer simulations strongly suggest that for random
permutations, a value of $m = 2n-1$ does not suffice. On the other
hand, we have constructed (details omitted from this manuscript)
oblivious $MC(n,2n-1)$ algorithms using permutations for $n=3$ and
$n=4$ (for the latter the proof of correctness is
computer-assisted). For $n \ge 5$ we have neither been able to
find such systems (not even in a fairly extensive computer search)
nor to rule out their existence.

We do not know how hard it is to recognize whether a given
collection of words constitute an oblivious $MC$ algorithm. This
can be viewed as the problem whether some digraph contains a
directed cycle or not. The point is that the digraph is presented
in a very compact form. It is not hard to place this problem in
PSPACE, but is it in a lower complexity class, such as co-NP or P?

\end{document}